\documentclass{amsart}

\usepackage{amsthm,amssymb,amsmath,amscd}

\theoremstyle{plain}
\newtheorem{thm}{Theorem}[section]
\newtheorem{lem}[thm]{Lemma}
\newtheorem{prop}[thm]{Proposition}

\theoremstyle{definition}
\newtheorem*{defn}{Definition}
\newtheorem*{ex}{Example}

\theoremstyle{remark}
\newtheorem*{rem}{Remark}
\newtheorem*{ack}{Acknowledgement}

\title[C$^*$-algebras generated by composition operators]
{C$^*$-algebras generated by multiplication operators and composition operators by functions with self-similar branches}
\author{Hiroyasu Hamada}

\address{National Institute of Technology, Sasebo College, 
Okishin, Sasebo, Nagasaki, 857-1193, Japan.}
\email{h-hamada@sasebo.ac.jp}



\begin{document}

\begin{abstract}
Let $K$ be a compact metric space and let $\varphi: K \to K$ be continuous.
We study C$^*$-algebra $\mathcal{MC}_\varphi$
generated by
all multiplication operators by continuous functions on $K$
and a composition operator $C_\varphi$ induced by $\varphi$ on
a certain $L^2$ space.
Let $\gamma = (\gamma_1, \dots, \gamma_n)$ be a system of proper contractions
on $K$. Suppose that $\gamma_1, \dots, \gamma_n$ are inverse branches of
$\varphi$ and $K$ is self-similar.
We consider the Hutchinson
measure $\mu^H$ of $\gamma$ and the $L^2$ space $L^2(K, \mu^H)$.
Then we show
that the C$^*$-algebra $\mathcal{MC}_\varphi$ is isomorphic to
the C$^*$-algebra $\mathcal{O}_\gamma (K)$ associated
with $\gamma$ under some conditions.
\end{abstract}

\maketitle

\section{Introduction}

Recently several authors considered C$^*$-algebras generated by
composition operators (and Toeplitz operators) to study properties of
composition operators or algebraic combinations of composition operators
and Toeplitz operators. Most of their studies have focused
on composition operators induced by linear fractional maps (\cite{J1, J2, KMM1, KMM3, KMM2, Pa, Q, Q2, SA}).
Watatani and the author \cite{HW}, and the author \cite{H1} considered 
C$^*$-algebras generated by
composition operators and Toeplitz operators for finite Blaschke products,
which are certain products of linear fractional maps.
Let $R$ be a finite Blaschke product of degree at least two. In \cite{H1}, 
we proved that there is a relation between
a C$^*$-algbara generated by a composition operator $C_R$
and Toeplitz operators and the C$^*$-algebra $\mathcal{O}_R(J_R)$ associated
with the complex dynamical system introduced in \cite{KW1}.

On the other hand, there are some studies on C$^*$-algebras generated by
composition operators on $L^2$ spaces, for example \cite{M} and \cite{H2}.
Matsumoto \cite{M} introduced some
C$^*$-algebras associated with cellular automata generated by composition
operators and multiplication operators.
Let $R$ be a rational function of degree at least two,
let $J_R$ be the Julia set of $R$ and let $\mu^L$ be the Lyubich measure
of $R$. In \cite{H2}, we studied the C$^*$-algebra $\mathcal{MC}_R$
generated by
all multiplication operators by continuous functions in $C(J_R)$
and the composition operator $C_R$ induced by $R$
on $L^2(J_R, \mu^L)$.
We showed that the C$^*$-algebra $\mathcal{MC}_R$ is isomorphic to
the C$^*$-algebra $\mathcal{O}_R (J_R)$ associated with the complex dynamical
system $\{R^{\circ n} \}_{n=1} ^\infty$.

More generally, we consider a C$^*$-algebra generated by
all multiplication operators
by continuous functions and a composition operator $C_\varphi$ induced by
$\varphi$ on a certain $L^2$ space.
Let $\varphi$ be the map $\varphi: [0,1] \to [0,1]$ 
defined by
\[
  \varphi (x) = \begin{cases}
                 2x & 0 \leq x \leq \frac{1}{2}, \\
                 -2x + 2 & \frac{1}{2} \leq x \leq 1.
                \end{cases}
\]
The map $\varphi$ is famous in dynamical system and is called the tent map.
Since $\varphi$ is not a rational map, we cannot adapt the theory of $\mathcal{MC}_R$ and $\mathcal{O}_R (J_R)$ in \cite{H2}.
In the same way as rational functions, we consider suitable C$^*$-algebras.
Kajiwara and Watatani \cite{KW2} also introduced the C$^*$-algebra
$\mathcal{O}_\gamma (K)$ associated with self-similar maps $\gamma$,
which is defined in a similar way to the C$^*$-algebra
$\mathcal{O}_R (J_R)$ associated with the complex dynamical
system $\{R^{\circ n} \}_{n=1} ^\infty$.

In this paper we consider the case that $\varphi$ is in a certain class
containing the tent map.
Let $(K, d)$ be a compact metric space, 
let $\gamma = (\gamma_1, \dots, \gamma_n)$ be a system of proper contractions
on $K$ and let $\varphi: K \to K$ be continuous.
Suppose that $\gamma_1, \dots, \gamma_n$ are inverse branches of
$\varphi$ and  $K$ is self-similar. We consider the Borel $\sigma$-algebra
$\mathcal{B}(K)$ on $K$ and the Hutchinson measure $\mu^H$ of $\gamma$.
The measure $\mu^H$ is the invariant measure of $\gamma$.
Let us denote by $\mathcal{MC}_\varphi$ the C$^*$-algebra generated by
multiplication operators $M_a$ for $a \in C(K)$
and the composition operator $C_\varphi$ on
$L^2 (K, \mathcal{B}(K), \mu^H)$.
Assume that the system
$\gamma = (\gamma_1, \dots, \gamma_n)$ satisfies the open set condition,
the finite branch condition and the measure separation condition in $K$. 
Then $\mathcal{MC}_\varphi$ is isomorphic to $\mathcal{O}_\gamma (K)$
associated with $\gamma$.

This means that we construct a representation of the C$^*$-algebra
$\mathcal{O}_\gamma (K)$ by multiplication operators composition
operators. We expect to be applied this result to analysis of the
C$^*$-algebra $\mathcal{O}_\gamma (K)$.

There are some remarks for $\mathcal{MC}_\varphi$.
We denote by $\mathcal{MC}_{\gamma_1, \gamma_2, \dots, \gamma_n}$
the C$^*$-algebra generated by all multiplication operators by continuous
functions and composition operators
$C_{\gamma_i}$ by $\gamma_i$ for $i = 1, 2, \dots, n$
on $L^2 (K, \mathcal{B}(K), \mu^H)$.
The definition of $\mathcal{MC}_\varphi$ is different from
that of $\mathcal{MC}_{\gamma_1, \gamma_2, \dots, \gamma_n}$.
Under some condition, we can show
$C_\varphi^* = \frac{1}{n} \sum_{i = 1} ^n C_{\gamma_i}$.
Thus $\mathcal{MC}_\varphi \subset \mathcal{MC}_{\gamma_1, \gamma_2, \dots, \gamma_n}$. Although $\mathcal{MC}_\varphi$ is not equal to $\mathcal{MC}_{\gamma_1, \gamma_2, \dots, \gamma_n}$ in general. For example, let $\gamma_1$ and $\gamma_2$ be the inverse branches of the tent map $\varphi$. Then $\mathcal{MC}_\varphi$ is isomorphic to the Cuntz algebra $\mathcal{O}_\infty$, while $\mathcal{MC}_{\gamma_1, \gamma_2}$ is isomorphic to the Cuntz algebra $\mathcal{O}_2$. Thus
 $\mathcal{MC}_\varphi$ is not equal to $\mathcal{MC}_{\gamma_1, \gamma_2}$.

\section{Covariant relations}

Let $(K,d)$ be a compact metric space. A continuous map $\gamma: K \to K$
is called a {\it proper contraction} if there exists constants
$0 < c_1 \leq c_2 < 1$ such that
\[
   c_1 d(x, y) \leq d(\gamma(x), \gamma(y)) \leq c_2 d(x, y), \quad x, y \in K.
\]

Let $\gamma = (\gamma_1, \dots, \gamma_n)$ be a family of proper contractions
on $(K,d)$. We say that $K$ is called {\it self-similar}
with respect to $\gamma$ if $K = \bigcup_{i = 1} ^n \gamma_i (K)$.
See \cite{F} and \cite{Ki} for more on fractal sets.

\begin{defn}
We say that $\gamma$ satisfies the {\it open set condition} in $K$
if there exists a non-empty open set $V \subset K$ such that
\[
   \bigcup_{i = 1} ^n \gamma_i(V) \subset V \quad \text{and} \quad
   \gamma_i (V) \cap \gamma_j (V) = \emptyset \quad \text{for} \quad
   i \neq j.
\]
\end{defn}

For a system $\gamma$ of proper contractions on a compact metric space $K$,
we introduce the following subsets of $K$.
\begin{align*}
B_\gamma &= \{ y \in K \, | \, y = \gamma_i (x) = \gamma_j (x)
\, \, \text{for some} \, \, x \in K \, \,  \text{and} \, \, i \neq j \}, \\
C_\gamma &= \{ x \in K \, | \, \gamma_i (x) = \gamma_j (x) \, \, 
\text{for some} \, \, i \neq j \}.
\end{align*}

\begin{defn}
We say that $\gamma$ satisfies the finite branch condition if $C_\gamma$
is finite set.
\end{defn}

In this paper, we consider $L^p$ spaces with respect to
Hutchinson measures. We recall the definition of
Hutchinson measures.

\begin{lem}[\cite{Hu}] \label{lem:Hutchinson measure}
Let $K$ be a compact metric space and let $\gamma$ be a system of proper
contractions. If $p_1, \dots, p_n \in \mathbb{R}$ satisfy
$\sum_{i=1} ^n p_i = 1$ and $p_i > 0$ for $i$,
then there exists a unique measure $\mu$ on $K$ such that
\[
   \mu(E) = \sum_{i=1}^n  p_i \mu (\gamma_i ^{-1}(E))
\]
for $E \in \mathcal{B}(K)$.
\end{lem}

\begin{defn}
We call the measure $\mu$ given by Lemma \ref{lem:Hutchinson measure}
the {\it self-similar measure} on $K$ with $\{p_i \}_{i = 1} ^n$.
In particular, we denote by $\mu^H$ the self-similar measure with
$p_i = \frac{1}{n}$ for $i$ and call this measure the Hutchinson measure.
\end{defn}

\begin{defn}[\cite{EKM}]
We say that $\gamma$ satisfies the {\it measure separation condition} in $K$
if $\mu (\gamma_i (K) \cap \gamma_j (K)) = 0$ for any self-similar measure $\mu$ and $i \neq j$.
\end{defn}

If $K \subset \mathbb{R}^d$, it is known that the open set condition is
equivalent to the measure separation condition. The theorem states that
many examples of systems of proper contractions satisfy
the measure separation condition.

\begin{thm}[\cite{Sc}]
Let $\gamma$ be a system of proper contractions. Assume
that $K \subset \mathbb{R}^d$ and $K$ is self-similar. Then the following
conditions are equivalent.
\begin{enumerate}
\item
$\gamma$ satisfies the open set condition in $K$.
\item
$\gamma$ satisfies the measure separation condition in $K$.
\end{enumerate}
\end{thm}

Let $\varphi:K \to K$ be measureable.
Suppose that $\gamma_1, \dots, \gamma_n$ are inverse
branches of $\varphi$, that is, $\varphi (\gamma_i (x)) = x$
for $x \in K$ and $i = 1, \dots, n$. Let $1 \leq p \leq \infty$.
We shall define the composition operator $C_\varphi$ on
$L^p (K, \mathcal{B}(K), \mu^H)$. 
The measurable function $\varphi$ induces a linear operator $C_\varphi$
from $L^p (K, \mathcal{B}(K), \mu^H)$ to the linear space of all
measurable functions on $(K, \mathcal{B}(K), \mu^H)$ defined as
$C_\varphi f = f \circ \varphi$ for $f \in L^p (K, \mathcal{B}(K), \mu^H)$.
If $C_\varphi : L^p (K, \mathcal{B}(K), \mu^H)
\to L^p (K, \mathcal{B}(K), \mu^H)$ is bounded, it is called
the {\it composition operator}
on $L^p (K, \mathcal{B}(K), \mu^H)$ induced by $\varphi$.

\begin{prop}
Let $\gamma = (\gamma_1, \dots, \gamma_n)$ be a system of proper contractions.
Assume that $K$ is self-similar and
the system $\gamma = (\gamma_1, \dots, \gamma_n)$ satisfies the
measure separation condition in $K$.
Then the operator $C_\varphi$ is an isometry on
$L^p (K, \mathcal{B}(K), \mu^H)$ for  $1 \leq p < \infty$.
\end{prop}

\begin{proof}
Since $\gamma_i$ is proper contraction, $\gamma_i: K \to \gamma_i(K)$
is bijective. Thus we have $(\mu^H \circ \gamma_i) (E) = \frac{1}{n}
\mu^H (E)$ for $E \in \mathcal{B}(K)$, where $\mu^H \circ \gamma_i$ is the
measure on $K$ defined by $(\mu^H \circ \gamma_i) (E)
= \mu^H (\gamma_i (E))$.
Since $K$ is self-similar and $\gamma$ satisfies the measure separation
condition in $K$, 
\begin{align*}
\| C_\varphi f \|_p ^p
&= \int_K | f(\varphi(x))|^p \, d \mu^H(x)
= \sum_{i = 1} ^n \int_{\gamma_i(K)} |f(\varphi(x))|^p \, d \mu^H(x) \\
&= \sum_{i = 1} ^n \int_K |f(y)|^p \, d (\mu^H \circ \gamma_i) (y)
= \sum_{i = 1} ^n \int_K |f(y)|^p \, \frac{1}{n} \, d \mu^H (y) \\
&= \int_K |f(y)|^p \, d \mu^H (y)
= \| f \|_p ^p
\end{align*}
for $f \in L^p (K, \mathcal{B}(K), \mu^H)$, which completes the proof.
\end{proof}

For $f \in L^1 (K, \mathcal{B}(K), \mu^H)$, we define an function $\mathcal{L}_\varphi f: K \to \mathbb{C}$
by
\[
   (\mathcal{L}_\varphi f)(x) = \frac{1}{n} \sum_{i = 1} ^n f(\gamma_i (x)),
   \quad x \in K.
\]
For $f \in C(K)$, we can easily see that $\mathcal{L}_\varphi f \in C(K)$ since $\gamma_1, \dots, \gamma_n$ are continuous functions.

\begin{lem} \label{lem:adjoint}
Let $\gamma = (\gamma_1, \dots, \gamma_n)$ be a system of proper contractions.
Assume that $K$ is self-similar and
the system $\gamma = (\gamma_1, \dots, \gamma_n)$ satisfies the
measure separation condition in $K$. Then $\mathcal{L}_\varphi$ is bounded on
$L^\infty (K, \mathcal{B}(K), \mu^H)$ and $C_\varphi ^* = \mathcal{L}_\varphi$,
where $C_\varphi$ is an operator on $L^1 (K, \mathcal{B}(K), \mu^H)$.
\end{lem}

\begin{proof}
Since $K$ is self-similar and $\gamma$ satisfies the measure separation
condition in $K$, we have
\begin{align*}
\langle C_\varphi ^* f, g \rangle &= \langle f, C_\varphi g \rangle
= \int_K f(x) g(\varphi(x)) \, d \mu^H(x)
= \sum_{i = 1} ^n \int_{\gamma_i(K)} f(x) g(\varphi(x)) \, d \mu^H(x) \\
&= \sum_{i = 1} ^n \int_K f(\gamma_i(y)) g(y) \, d (\mu^H \circ \gamma_i) (y)
= \sum_{i = 1} ^n \int_K f(\gamma_i(y)) g(y) \, \frac{1}{n} \, d \mu^H (y) \\
&= \int_K (\mathcal{L}_\varphi f)(y) g(y) \, d \mu^H (y)
= \langle \mathcal{L}_\varphi  f, g \rangle
\end{align*}
for $f \in L^\infty (K, \mathcal{B}(K), \mu^H)$
and $g \in  L^1 (K, \mathcal{B}(K), \mu^H)$,
which completes the proof.
\end{proof}

For $a \in L^\infty (K, \mathcal{B}(K), \mu^H)$, we define
the multiplication operator $M_a$ on
$L^2 (K, \mathcal{B}(K), \mu^H)$ by
$M_a f = a f$ for $f \in L^2 (K, \mathcal{B}(K), \mu^H)$.
We have the following covariant relation by the same
argument in the proof of \cite[Proposition 2.3]{H2}.

\begin{prop} \label{prop:covariant}
Let $\gamma = (\gamma_1, \dots, \gamma_n)$ be a system of proper contractions.
Assume that $K$ is self-similar and
the system $\gamma = (\gamma_1, \dots, \gamma_n)$ satisfies the
measure separation condition in $K$.
Let
$C_\varphi$ be the composition operator on $L^2 (K, \mathcal{B}(K), \mu^H)$
with $\varphi$. Then we have
\[
  C_\varphi ^* M_a C_\varphi = M_{\mathcal{L}_\varphi (a)}
\]
for $a \in L^\infty (K, \mathcal{B}(K), \mu^H)$.
\end{prop}

\begin{proof}
For $f, g \in L^2 (K, \mathcal{B}(K), \mu^H)$,
we have
\begin{align*}
  \langle C_\varphi ^* M_a C_\varphi f, g \rangle
  &= \langle M_a C_\varphi f,  C_\varphi g \rangle
  = \int_K a (f \circ \varphi) 
    \overline {(g \circ \varphi)} d \mu^H \\
  &= \int_K a C_\varphi (f \overline{g}) d \mu^H
  = \int_K \mathcal{L}_\varphi (a) f \overline{g} d \mu^H
  = \langle M_{\mathcal{L}_\varphi (a)} f, g \rangle
\end{align*}
by Lemma \ref{lem:adjoint}, where
$C_\varphi$ is also regarded as the composition operator
on $L^1(K, \mathcal{B}(K), \mu^H)$.
\end{proof}

\section{C$^*$-algebras associated with self-similar sets}

We recall the construction of Cuntz-Pimsner algebras \cite{Pi} (see also
\cite{K}). 
Let $A$ be a C$^*$-algebra and let $X$ be a right Hilbert $A$-module.
A sequence $\{ u_i \}_{i=1} ^\infty$
of $X$ is called a {\it countable basis} of X if
$\xi = \sum_{i=1} ^\infty u_i \langle u_i, \xi \rangle_A$
for $\xi \in X$, where the right hand side converges in norm.
We denote by $\mathcal{L}(X)$ the C$^*$-algebra of the adjointable bounded operators 
on $X$.  
For $\xi$, $\eta \in X$, the operator $\theta _{\xi,\eta}$
is defined by $\theta _{\xi,\eta}(\zeta) = \xi \langle \eta, \zeta \rangle_A$
for $\zeta \in X$. 
The closure of the linear span of these operators is denoted by $\mathcal{K}(X)$. 
We say that 
$X$ is a {\it Hilbert bimodule} (or {\it C$^*$-correspondence}) 
over $A$ if $X$ is a right Hilbert $A$-module 
with a $*$-homomorphism $\phi : A \rightarrow \mathcal{L}(X)$.
We always assume that $\phi$ is injective. 

A {\it representation} of the Hilbert bimodule $X$
over $A$ on a C$^*$-algebra $D$
is a pair $(\rho, V)$ constituted by a $*$-homomorphism $\rho: A \to D$ and
a linear map $V: X \to D$ satisfying
\[
  \rho(a) V_\xi = V_{\phi(a) \xi}, \quad
  V_\xi ^* V_\eta = \rho( \langle \xi, \eta \rangle_A)
\]
for $a \in A$ and $\xi, \eta \in X$.
It is known that $V_\xi \rho(b) = V_{\xi b}$ follows
automatically (see for example \cite{K}).
We define a $*$-homomorphism $\psi_V : \mathcal{K}(X) \to D$
by $\psi_V ( \theta_{\xi, \eta}) = V_{\xi} V_{\eta}^*$ for $\xi, \eta \in X$
(see for example \cite[Lemma 2.2]{KPW}).
A representation $(\rho, V)$ is said to be {\it covariant} if
$\rho(a) = \psi_V(\phi(a))$ for all $a \in J(X)
:= \phi ^{-1} (\mathcal{K}(X))$.
Suppose the Hilbert bimodule $X$
has a countable basis $\{u_i \}_{i=1} ^\infty$ and $(\rho, V)$ is
a representation of $X$.
Then $(\rho, V)$ is covariant if and only if
$\| \sum_{i=1} ^n \rho(a) V_{u_i} V_{u_i} ^* - \rho(a) \| \to 0$ as
$n \to \infty$ for $a \in J(X)$, since $\{ \sum_{i=1} ^n \theta_{u_i, u_i} \}_{n=1} ^\infty$ is an approximate unit for $\mathcal{K}(X)$.

Let $(i, S)$ be the representation of $X$ which is universal for
all covariant representations. 
The {\it Cuntz-Pimsner algebra} ${\mathcal O}_X$ is 
the C$^*$-algebra generated by $i(a)$ with $a \in A$ and 
$S_{\xi}$ with $\xi \in X$.
We note that $i$ is known to be injective
\cite{Pi} (see also \cite[Proposition 4.11]{K}).
We usually identify $i(a)$ with $a$ in $A$.

Let $\gamma = (\gamma_1, \dots, \gamma_n)$ be a system of proper contractions
on a compact metric space $K$.
Let $A = C(K)$ and $Y = C(\mathcal{C})$,
where $\mathcal{C} = \bigcup_{i = 1} ^n \{ (\gamma_i(y), y) \, | \, y
\in K \}$  is the cograph of $\gamma_i$.
Then $Y$ is an $A$-$A$ bimodule over $A$ by 
\[
  (a \cdot f \cdot b)(\gamma_i(y),y) = a(\gamma_i(y)) f(\gamma_i(y),y) b(y),\quad a, b \in A, \, 
  f \in Y.
\]
We define an $A$-valued inner product $\langle \ , \ \rangle_A$ on $Y$ by 
\[
  \langle f, g \rangle_A (y) = \sum _{i=1} ^n
  \overline{f(\gamma_i (y), y)} g(\gamma_i(y), y),
  \quad f, g \in Y, \, y \in K.
\]
Then $Y$ is a Hilbert bimodule over $A$. 
The C$^*$-algebra ${\mathcal O}_\gamma (K)$
is defined as the Cuntz-Pimsner algebra of the Hilbert bimodule 
$Y = C(\mathcal{C})$ over $A = C(K)$.

\section{Main theorem}

\begin{defn}
Let $\varphi : K \to K$ be continuous.
Suppose that composition operator $C_\varphi$ on
$L^2(K, \mathcal{B}(K), \mu^{H})$ is bounded.
We denote by $\mathcal{MC}_\varphi$ the C$^*$-algebra
generated by all multiplication operators by
continuous functions in $C(K)$
and the composition operator $C_\varphi$ on $L^2(K, \mathcal{B}(K), \mu^{H})$.
\end{defn}

Let $\gamma = (\gamma_1, \dots, \gamma_n)$ be a system of proper contractions on $K$. Suppose that $\gamma_1, \gamma_2, \dots, \gamma_n$ are inverse branches of $\varphi$ and $K$ is self-similar. In this section we shall show that the C$^*$-algebra $\mathcal{MC}_\varphi$
is isomorphic to the C$^*$-algebra $\mathcal{O}_\gamma (K)$
under some assumptions.

\begin{rem}
We denote by $\mathcal{MC}_{\gamma_1, \gamma_2, \dots, \gamma_n}$
the C$^*$-algebra generated by all multiplication operators by continuous
functions in $C(K)$ and composition operators
$C_{\gamma_i}$ by $\gamma_i$ for $i = 1, 2, \dots, n$
on $L^2 (K, \mathcal{B}(K), \mu^H)$.
The definition of $\mathcal{MC}_\varphi$ is different from
that of $\mathcal{MC}_{\gamma_1, \gamma_2, \dots, \gamma_n}$.
If $\gamma$ satisfies the measure separation condition in $K$, then
we have $C_\varphi^* = \frac{1}{n} \sum_{i = 1} ^n C_{\gamma_i}$
by Lemma \ref{lem:adjoint}.
Thus $\mathcal{MC}_\varphi \subset \mathcal{MC}_{\gamma_1, \gamma_2, \dots, \gamma_n}$. Although $\mathcal{MC}_\varphi$ is not equal to $\mathcal{MC}_{\gamma_1, \gamma_2, \dots, \gamma_n}$ in general. For example,
let $\gamma_1$ and $\gamma_2$ be the inverse branches of
the tent map $\varphi$.
Then $\mathcal{MC}_\varphi$ is not equal to
$\mathcal{MC}_{\gamma_1, \gamma_2}$. We shall consider this case in Section \ref{sec:Ex}.
\end{rem}

Let $\varphi:K \to K$ be continuous.
Let $A = C(K)$ and $X = C(K)$.
Then $X$ is an $A$-$A$ bimodule over $A$ by 
\[
   (a \cdot \xi \cdot b) (x) = a(x) \xi(x) b(\varphi(x)) \quad a,b \in A, \, 
  \xi \in X.
\]
We define an $A$-valued inner product $\langle \ , \ \rangle_A$ on $X$ by
\[
  \langle \xi, \eta \rangle _A (x) = \frac{1}{n} \sum_{i = 1} ^n
  \overline{\xi(\gamma_i (x))} \eta(\gamma_i(x))
  \, \,
  \left (\, = (\mathcal{L}_\varphi (\overline{\xi} \eta))(x) \, \right ), \quad
  \xi, \eta \in X.
\]
Then $X$ is a Hilbert bimodule over $A$.
Put $\| \xi \|_2 = \| \langle \xi, \xi \rangle_A \|_\infty ^{1/2}$
for $\xi \in X$,
where $\| \ \|_\infty$ is the sup norm on $K$.
Let $\Phi: Y \to X$ be defined by
$(\Phi(f))(x) = \sqrt{n} f(x, \varphi(x))$ for $f \in Y$.
It is easy to see that $\Phi$ is an isomorphism and $X$ is
isomorphic to $Y$ as Hilbert bimodules over $A$.
Hence the C$^*$-algebra $\mathcal{O}_\gamma (K)$ is isomorphic to
the Cuntz-Pimsner algebra $\mathcal{O}_X$ constructed from $X$.

We need some analyses based on bases of the Hilbert bimodule $X$
to show an equation containing the composition operator $C_\varphi$
and multiplication operators.

\begin{lem} \label{lem:key_lem}
Let $u_1, \dots, u_N \in X$ and
let $\gamma = (\gamma_1, \dots, \gamma_n)$ be a system of proper contractions.
Assume that $K$ is self-similar and
the system $\gamma = (\gamma_1, \dots, \gamma_n)$ satisfies the
measure separation condition in $K$.
Then
\[
  \sum_{i=1} ^N M_{u_i} C_\varphi C_\varphi ^* M_{u_i} ^* a
  = \sum_{i=1} ^N u_i \cdot \langle u_i , a \rangle_A
\]
for $a \in A$.
\end{lem}

\begin{proof}
Since $a = M_a C_\varphi 1$, we have
\begin{align*}
\sum_{i=1} ^N M_{u_i} C_\varphi C_\varphi ^* M_{u_i} ^* a
&= \sum_{i=1} ^N M_{u_i} C_\varphi C_\varphi ^* M_{u_i} ^* M_a C_\varphi 1 \\
&= \sum_{i=1} ^N M_{u_i} C_\varphi C_\varphi ^* M_{\overline u_i a} C_\varphi 1 \\
&= \sum_{i=1} ^N M_{u_i} C_\varphi M_{\mathcal{L}_\varphi (\overline u_i a)} 1
   \quad \quad \text{by Proposition \ref{prop:covariant}} \\
&= \sum_{i=1} ^N M_{u_i} M_{\mathcal{L}_\varphi (\overline u_i a) \circ \varphi} C_\varphi 1 \\
&= \sum_{i=1} ^N u_i \mathcal{L}_\varphi (\overline u_i a) \circ \varphi \\
&= \sum_{i=1} ^N u_i \cdot \langle u_i , a \rangle_A,
\end{align*}
which completes the proof.
\end{proof}

\begin{lem} \label{lem:norm}
Let $\{ u_i \}_{i=1} ^\infty$ be a countable basis of $X$ and
let $\gamma = (\gamma_1, \dots, \gamma_n)$ be a system of proper contractions.
Assume that $K$ is self-similar and
the system $\gamma = (\gamma_1, \dots, \gamma_n)$ satisfies the
measure separation condition in $K$.
Then
\[
  0 \leq \sum_{i=1} ^N M_{u_i} C_\varphi C_\varphi ^* M_{u_i} ^*  \leq I. 
\]
\end{lem}

\begin{proof}
Set $T_N = \sum_{i=1} ^N M_{u_i} C_\varphi C_\varphi ^* M_{u_i} ^*$.
It is clear that $T_N$ is a positive operator.
We shall show $T_N \leq I$. By Lemma \ref{lem:key_lem},
\[
  \langle T_N f, f \rangle = \int_K (T_N f) (x)
  \overline{f(x)} d \mu ^H (x)
  = \int_K \left( \sum_{i=1} ^N u_i \cdot
  \langle u_i, f \rangle_A \right ) (x) \overline{f(x)} d \mu ^H (x)
\]
for $f \in C(K)$.
Since $\{ u_i \}_{i=1} ^\infty$ is a countable basis of $X$, 
for $f \in C(K)$, we have
$\sum_{i=1} ^N u_i \cdot \langle u_i, f \rangle_A \to f$
with respect to $\| \, \, \|_2$ as $N \to \infty$.
Since the two norms $\| \, \, \|_2$ and $\| \, \, \|_\infty$ are
equivalent (see the proof of \cite[Proposition 2.1]{KW2}),
$\sum_{i=1} ^N u_i \cdot \langle u_i, f \rangle_A$ converges to $f$
with respect to $\| \, \, \|_\infty$.
Thus
\[
  \langle T_N f, f \rangle \to \int_K f(x) \overline{f(x)} d \mu ^H (x) = \langle f, f \rangle \quad {\text as} \, \, \, N \to \infty
\]
for $f \in C(K)$.
Therefore $\langle T_N f, f \rangle \leq \langle f, f \rangle$ for $f \in C(K)$.
Since the Hutchinson measure $\mu^H$ on $K$
is regular, $C(K)$ is dense in $L^2 (K, \mathcal{B}(K), \mu^{H})$. Hence
we have $T_N \leq I$. This completes the proof.
\end{proof}

We now recall a description of the ideal $J(X)$ of $A$.
By \cite[Proposition 2.6]{KW2}, we can write
$J(X) = \{ a \in A \, | \,
a \, \, \text{vanishes on} \, \, B_\gamma \}$.
We define a subset $J(X)^0$ of $J(X)$
by $J(X) ^0 = \{ a \in A \, | \,
a \, \, \text{vanishes on} \, \, B_\gamma \, \,
\text{and has compact support on } K \smallsetminus
B_\gamma \}$.
Then $J(X)^0$ is dense in $J(X)$.

\begin{lem} \label{lem:countable_basis}
Let $\gamma = (\gamma_1, \dots, \gamma_n)$ be a system of proper contractions.
Assume that $K$ is self-similar and the system
$\gamma = (\gamma_1, \dots, \gamma_n)$ satisfies the finite branch
condition and the measure separation condition in $K$. Then
there exists a countable basis $\{ u_i \}_{i=1} ^\infty$ of $X$
such that
\[
  \sum_{i=1} ^\infty M_a M_{u_i} C_\varphi C_\varphi ^* M_{u_i} ^* = M_a
\]
for $a \in J(X)$.
\end{lem}

\begin{proof}
Since $\gamma$ satisfies the finite branch condition,
there exists a countable basis $\{ u_i \}_{i=1} ^\infty$ of $X$
satisfying the following property by \cite[Subsection 3.2]{Kaj}.
For any $b \in J (X) ^0$, there exists $M>0$  such that
${\rm supp} \, b \cap {\rm supp} \, u_m = \emptyset$ for $m \geq M$.
Since $J (X) ^0$ is dense in $J (X)$,
for any $a \in A$ and any $\varepsilon > 0$, there exists
$b \in J (X) ^0$
such that $\| a - b \| < \varepsilon / 2$.
Let $m \geq M$.
Then by Lemma \ref{lem:key_lem} and $b u_i = 0$ for
$i \geq m$, it follows that
\[
   \sum_{i=1} ^m M_b M_{u_i} C_\varphi C_\varphi ^* M_{u_i} ^* f
   = \sum_{i=1} ^m b u_i \cdot \langle u_i, f \rangle_A
   = \sum_{i=1} ^\infty b u_i \cdot \langle u_i, f \rangle_A
   = bf = M_b f
\]
for $f \in C(K)$. Since
$C(K)$ is dense in $L^2(K, \mathcal{B}(K), \mu^H)$, we have
\[
  \sum_{i=1} ^m M_b M_{u_i} C_\varphi C_\varphi ^* M_{u_i} ^* = M_b.
\]
From Lemma \ref{lem:norm} it follows that
\begin{align*}
\left \| \sum_{i=1} ^m M_a M_{u_i} C_\varphi C_\varphi ^* M_{u_i} ^* - M_a \right \|
& \leq
\left \| \sum_{i=1} ^m M_a M_{u_i} C_\varphi C_\varphi ^* M_{u_i} ^*
- \sum_{i=1} ^m M_b M_{u_i} C_\varphi C_\varphi ^* M_{u_i} ^* \right \| \\
& \quad \quad +
\left \| \sum_{i=1} ^m M_b M_{u_i} C_\varphi C_\varphi ^* M_{u_i} ^* - M_b
\right \| + \| M_b - M_a \| \\
& \leq  \| M_a - M_b \| \, \left \| \sum_{i=1} ^m M_{u_i} C_\varphi C_\varphi ^* M_{u_i} ^*
\right \| + \| M_a - M_b \| \\
& < \frac{\varepsilon}{2} + \frac{\varepsilon}{2} = \varepsilon,
\end{align*}
which completes the proof.
\end{proof}

The following theorem is the main result of the paper.

\begin{thm} \label{thm:main}
Let $(K, d)$ be a compact metric space, 
let $\gamma = (\gamma_1, \dots, \gamma_n)$ be a system of proper contractions
on $K$ and let $\varphi: K \to K$ be continuous.
Suppose that $\gamma_1, \dots, \gamma_n$ are inverse branches of
$\varphi$. Assume that $K$ is self-similar and the system
$\gamma = (\gamma_1, \dots, \gamma_n)$ satisfies the open set condition,
the finite branch condition and the measure separation condition in $K$. 
Then $\mathcal{MC}_\varphi$ is isomorphic to $\mathcal{O}_\gamma (K)$.
\end{thm}

\begin{proof}
Put $\rho(a) = M_a$ and $V_\xi = M_\xi C_\varphi$ for $a \in A$ and $\xi \in X$.
Then we have
\[
  \rho(a) V_\xi = M_a M_\xi C_\varphi = M_{a \xi} C_\varphi =V_{a \cdot \xi}
\]
and
\[
  V_\xi ^* V_\eta = C_\varphi ^* M_\xi ^* M_\eta C_\varphi 
  = C_\varphi ^* M_{\overline{\xi} \eta} C_\varphi
  = M_{\mathcal{L}_\varphi (\overline{\xi} \eta)}
  = \rho (\mathcal{L}_\varphi (\overline{\xi} \eta))
  = \rho ( \langle \xi, \eta \rangle_A)
\]
for $a \in A$ and $\xi, \eta \in X$
by Proposition \ref{prop:covariant}.
Let $\{ u_i \} _{i=1} ^\infty$ be a countable basis of $X$.
Then, applying Lemma \ref{lem:countable_basis},
\[
  \sum_{i=1} ^\infty \rho(a) V_{u_i} V_{u_i} ^*
  = \sum_{i=1} ^\infty M_a M_{u_i} C_\varphi C_\varphi ^* M_{u_i} ^*
  = M_a = \rho(a)
\]
for $a \in J (X)$.
Since the support of the Hutchinson measure $\mu^H$ is
the self-similar set $K$, the $*$-homomorphism $\rho$ is injective.
By the universality and the simplicity of $\mathcal{O}_\gamma (K)$
(\cite[Theorem 3.8]{KW2}),
the C$^*$-algebra $\mathcal{MC}_\varphi$ is isomorphic to
$\mathcal{O}_\gamma (K)$.
\end{proof}

\section{Examples} \label{sec:Ex}

We give some examples for C$^*$-algebras generated by a composition
operator $C_\varphi$ and multiplication operators.

\begin{ex}
A tent map $\varphi: [0,1] \to [0,1]$ is defined by
\[
  \varphi (x) = \begin{cases}
                 2x & 0 \leq x \leq \frac{1}{2}, \\
                 -2x + 2 & \frac{1}{2} \leq x \leq 1.
                \end{cases}
\]
Let
\[
   \gamma_1 (y) = \frac{1}{2} y \quad \text{and} \quad
   \gamma_2 (y) = - \frac{1}{2} y + 1.
\]
Then $\gamma_1$ and $\gamma_2$ are inverse branches of $\varphi$ and
$K = [0,1]$ is the self-similar set with respect to
$\gamma = (\gamma_1, \gamma_2)$. 
The Hutchinson measure $\mu^H$ on $[0,1]$ coincides with
the Lebesgue measure $m$ on $[0,1]$. 
The system $\gamma$ satisfies
the open set condition, the finite branch condition and the measure
separation condition in $K$.
We consider the composition operator $C_\varphi$ on
$L^2 ([0,1], \mathcal{B}([0, 1]), m)$. By Theorem \ref{thm:main},
the C$^*$-algebra $\mathcal{MC}_\varphi$ is isomorphic to
$\mathcal{O}_\gamma ([0,1])$. Moreover 
$\mathcal{O}_\gamma ([0,1])$ is isomorphic to
the Cuntz algebra $\mathcal{O}_\infty$ by \cite[Example 4.5]{KW2}.
Thus $\mathcal{MC}_\varphi$  is isomorphic to $\mathcal{O}_\infty$.
\end{ex}

\begin{rem}
Let $\gamma = (\gamma_1, \dots, \gamma_n)$ be a system of proper contractions on $K$. Assume that $K$ is self-similar and
the system $\gamma = (\gamma_1, \dots, \gamma_n)$ satisfies the
measure separation condition in $K$.
Then $\mathcal{MC}_\varphi \subset
\mathcal{MC}_{\gamma_1, \gamma_2, \dots, \gamma_n}$.
Although $\mathcal{MC}_\varphi$ is not equal to $\mathcal{MC}_{\gamma_1, \gamma_2, \dots, \gamma_n}$ in general.
In the above example, we can see that
$\mathcal{MC}_{\gamma_1, \gamma_2}$ is isomorphic to the Cuntz algebra
$\mathcal{O}_2$ by \cite{H3}.
Thus $\mathcal{MC}_\varphi$ is not equal
to $\mathcal{MC}_{\gamma_1, \gamma_2}$.
\end{rem}

\begin{ex}
Let $K = \{1, \dots , n \}^{\mathbb{N}}$. The space $K$ is the space
of one-sided sequences $w = \{ w_i \}_{i=1} ^\infty$ of
$\{ 1, \dots, n \}$. Let $\varphi: K \to K$ be the shift
\[
   \varphi( w_1, w_2, \dots ) = ( w_2, w_3, \dots ).
\]
Then inverse branches of $\varphi$ are $\gamma_1, \dots, \gamma_n$ such
that
\[
   \gamma_i ( w_1, w_2, \dots ) = ( i, w_1, w_2, \dots )
\] 
for $i$. We define a metric $d$ on $K$ by
\[
   d(w, v) = \sum_{i = 1} ^\infty \frac{1}{2^i} \, (1 - \delta_{w_i, v_i})
\]
for $w = \{ w_i \}_{i=1} ^\infty, \, v = \{ v_i \}_{i=1} ^\infty \in K$.
Then $\gamma_i, \dots, \gamma_n$ are proper contractions with
the Lipschitz constant $\frac{1}{2}$ and $K$ is the self-similar set
with respect to $\gamma = (\gamma_1, \dots, \gamma_n)$.
The Hutchinson measure $\mu^H$ on $K$ coincides with
the product measure of the discrete probability measure $\nu$ on
$\{1, \dots , n \}$ such that $\nu(\{ i \}) = \frac{1}{n}$ for
$i = 1, \dots, n$. The system $\gamma$ satisfies
the open set condition, the finite branch condition and the measure
separation condition in $K$.
We consider the composition operator $C_\varphi$ on
$L^2 (K, \mathcal{B}(K), \mu^H)$. By Theorem \ref{thm:main},
the C$^*$-algebra $\mathcal{MC}_\varphi$ is isomorphic to
$\mathcal{O}_\gamma (K)$. Moreover 
$\mathcal{O}_\gamma (K)$ is isomorphic to
the Cuntz algebra $\mathcal{O}_n$ by \cite[Example 4.2]{KW2}.
Thus $\mathcal{MC}_\varphi$  is isomorphic to $\mathcal{O}_n$.

\end{ex}

\begin{ack}
The author wishes to express his thanks to Professor Yasuo Watatani and
Professor Tsuyoshi Kajiwara for suggesting the problem and for
many stimulating conversations.
\end{ack}


\begin{thebibliography}{99}

\bibitem{EKM}
M. Elekes, T. Keleti and A. M\'{a}th\'{e},
{\it Self-similar and self-affine sets: measure of the intersection of two
copies},
Ergodic Theory Dynam. Systems {\bf 30} (2010), 399--440. 

\bibitem{F}
K. J. Falconer,
{\it Fractal Geometry},
Wiley, Chichester, 1997.

\bibitem{H1}
H. Hamada,
{\it Quotient algebras of Toeplitz-composition C$^*$-algebras
for finite Blaschke products}, Complex Anal. Oper. Theory {\bf 8} (2014), 843--862.

\bibitem{H2}
H. Hamada,
{\it C$^*$-algebras generated by multiplication operators and
composition operators with rational functions},
J. Operator Theory {\bf 75} (2016), 289--298.

\bibitem{H3}
H. Hamada,
{\it C$^*$-algebras generated by multiplication operators and
composition operators with self-similar maps},
in preparation.

\bibitem{HW}
H. Hamada and Y. Watatani,
{\it Toeplitz-composition C$^*$-algebras for certain finite Blaschke products},
Proc. Amer. Math. Soc. {\bf 138} (2010), 2113--2123.

\bibitem{Hu}
J. Huchinson,
{\it Fractals and self-similarity},
Indiana Univ. Math. J. {\bf 30} (1981), 713--747.

\bibitem{J1}
M. T. Jury,
{\it The Fredholm index for elements of Toeplitz-composition C$^*$-algebras},
Integral Equations Operator Theory {\bf 58} (2007), 341--362.

\bibitem{J2}
M. T. Jury,
{\it C$^*$-algebras generated by groups of composition operators},
Indiana Univ. Math. J. {\bf 56} (2007), 3171--3192.

\bibitem{Kaj}
T, Kajiwara,
{\it Countable bases for Hilbert C$^*$-modules
and classification of KMS states},
Operator structures and dynamical systems, 73--91, Contemp. Math., {\bf 503},
Amer. Math. Soc., Providence, RI, 2009.

\bibitem{KPW}
T. Kajiwara, C. Pinzari and Y. Watatani,
{\it Ideal structure and simplicity of the C$^*$-algebras generated
by Hilbert bimodules},
J. Funct. Anal. {\bf 159} (1998), 295--322.

\bibitem{KW1}
T. Kajiwara and Y. Watatani,
{\it C$^*$-algebras associated with complex dynamical systems},
Indiana Math. J. {\bf 54} (2005), 755--778.

\bibitem{KW2}
T. Kajiwara and Y. Watatani,
{\it C$^*$-algebras associated with self-similar sets},
J. Operator Theory {\bf 56} (2006), 225--247.

\bibitem{K}
T. Katsura,
{\it On C$^*$-algebras associated with C$^*$-correspondences},
J. Funct. Anal. {\bf 217} (2004), 366--401.

\bibitem{Ki}
J. Kigami,
{\it Analysis on Fractals},
Cambridge University Press, Cambridge, 2001.

\bibitem{KMM1}
T. L. Kriete, B. D. MacCluer and J. L. Moorhouse,
{\it Toeplitz-composition C$^*$-algebras},
J. Operator Theory {\bf 58} (2007), 135--156.

\bibitem{KMM3}
T. L. Kriete, B. D. MacCluer and J. L. Moorhouse,
{\it Spectral theory for algebraic combinations of Toeplitz
and composition operator},
J. Funct. Anal. {\bf 257} (2009), 2378--2409.

\bibitem{KMM2}
T. L. Kriete, B. D. MacCluer and J. L. Moorhouse,
{\it Composition operators within singly generated composition C$^*$-algebras},
Israel J. Math. {\bf 179} (2010), 449--477.

\bibitem{M}
K. Matsumoto,
{\it C$^*$-algbras associated with cellular automata},
Math. Scand. {\bf 75} (1994), 195--216.

\bibitem{Pa}
E. Park,
{\it Toeplitz algebras and extensions of irrational rotation algebras},
Canad. Math. Bull. {\bf 48} (2005), 607--613.

\bibitem{Pi}
M. V. Pimsner,
{\it A class of C$^*$-algebras generating both Cuntz-Krieger algebras and
crossed product by $\mathbb{Z}$ },
Free Probability Theory, Fields Inst. Commun., Vol 12, Amer. Math.
Soc., Providence, RI, pp. 189--212.

\bibitem{Q}
K. S. Quertermous,
{\it A semigroup composition C$^*$-algebra},
J. Operator Theory {\bf 67} (2012), 581--604.

\bibitem{Q2}
K. S. Quertermous,
{\it Fixed point composition and Toeplitz-composition C$^*$-algebras},
J. Funct. Anal. {\bf 265} (2013), 743--764.

\bibitem{SA}
M. K. Sarvestani and M. Amini,
{\it The C$^*$-algebra generated by irreducible
Toeplitz and composition operators},
Rocky Mountain J. Math.@{\bf 47} (2017), 1301--1316.

\bibitem{Sc}
A. Schief,
{\it Separation properties for self-similar sets},
Proc. Amer. Math. Soc. {\bf 122} (1994), 111--115.


\end{thebibliography}
\end{document}